\documentclass[12pt,a4paper]{article}
\usepackage[utf8x]{inputenc}
\usepackage{ucs}
\usepackage{amsmath}
\usepackage{amsfonts}
\usepackage{amssymb}
\usepackage{graphicx}
\usepackage{mathrsfs} 
\usepackage{url} 
\usepackage{xcolor}
\usepackage{amsthm}

\newcommand{\R}{{\mathbb R}}
\newcommand{\Q}{{\mathbb Q}}

\newcommand{\N}{{\mathbb N}}
\newcommand{\Z}{{\mathbb Z}}

\def\EMdash{\leavevmode\hbox to 7.5mm{\vrule height .63ex depth -.59ex
    width 5.4mm\hfill}}

\newcommand{\Phibar}{\overline{\Phi}}

\numberwithin{equation}{section}

\newtheorem{thm}{Theorem}[section]
\newtheorem{lemma}[thm]{Lemma}
\newtheorem{coll}[thm]{Corollary}

\title{Summation of certain infinite Fibonacci related series}
\author{Bakir Farhi \\[1mm]
Laboratoire de Mathématiques appliquées \\
Faculté des Sciences Exactes \\
Université de Bejaia, 06000 Bejaia, Algeria \\[1mm]
bakir.farhi@gmail.com
}
\date{}

\begin{document}
\maketitle

\begin{abstract}
In this paper, we find the closed sums of certain type of Fibonacci related convergent series. In particular, we generalize some results already obtained by Brousseau, Popov, Rabinowitz and others. 
\end{abstract}

\section{Introduction}
Throughout this paper, we let $\N^*$ denote the set $\N \setminus \{0\}$ of positive integers. We denote by $\Phi$ the golden ration ($\Phi := \frac{1 + \sqrt{5}}{2}$) and by $\Phibar$ its conjugate in the quadratic field $\Q(\sqrt{5})$; that is $\Phibar := \frac{1 - \sqrt{5}}{2} = - \frac{1}{\Phi}$.

The Fibonacci sequence ${(F_n)}_{n \in \N}$ is defined by: $F_0 = 0$, $F_1 = 1$ and for all $n \in \N$:
\begin{equation}\label{eq 1.1}
F_{n + 2} = F_n + F_{n + 1}
\end{equation}
This sequence can be extended to the negative index $n$ by rewriting the recurrence relation \eqref{eq 1.1} as: $F_n = F_{n + 2} - F_{n + 1}$. By induction, we easily show that for all $n \in \Z$, we have:
\begin{equation}\label{eq 1.2}
F_{- n} = (-1)^{n + 1} F_n
\end{equation}
A closed formula of $F_n$ ($n \in \Z$) in terms of $n$ is given by:
\begin{equation}\label{eq 1.3}
F_n = \frac{1}{\sqrt{5}} \left(\Phi^n - \Phibar^n\right)
\end{equation}
Formula \eqref{eq 1.3} is called \emph{the Binet Formula} (see e.g., \cite[Chapter 8]{hons}). \\
Another sequence ultimately related to the Fibonacci sequence is the Lucas sequence ${(L_n)}_{n \in \Z}$, which is defined by: $L_0 = 2$, $L_1 = 1$ and for all $n \in \Z$:
\begin{equation}\label{eq 1.4}
L_{n + 2} = L_n + L_{n + 1}
\end{equation}
The connections and likenesses between the Fibonacci and the Lucas sequences are numerous; among them, we just cite the followings:
\begin{eqnarray}
L_n & = & F_{n - 1} + F_{n + 1} \label{eq 1.5} \\
L_n & = & \Phi^n + \Phibar^n \label{eq 1.6} \\
F_{2 n} & = & F_n L_n \label{eq 1.7}
\end{eqnarray} 
which hold for any $n \in \Z$. \\
Using only the Binet formula \eqref{eq 1.3}, several properties and formulas related to the Fibonacci sequence can be found. We show for example that for all $(n , m , r) \in \Z^3$, we have:
\begin{equation}\label{eq 1.8}
F_r F_{n - m} = (-1)^m \left(F_n F_{m + r} - F_m F_{n + r}\right)
\end{equation}
We call Formula \eqref{eq 1.8} \emph{the generalized subtraction formula}, where \emph{the subtraction formula} is that which is obtained by taking $r = 1$ in \eqref{eq 1.8}. Note that Formula \eqref{eq 1.8} plays a vital role in Section 2. Using Binet's formula, we also show that for all $n , r \in \Z$, we have:
\begin{eqnarray}
\Phi^n F_r & = & \Phi^r F_n + (-1)^{r + 1} F_{n - r} \label{eq 1.9} \\
\Phibar^n F_r & = & \Phibar^r F_n + (-1)^{r + 1} F_{n - r} \label{eq 1.10}
\end{eqnarray}   
These two last formulas will be used in Section 3.

Fibonacci's sequence plays a very important role in theoretical and applied mathematics. Particularly, the summation of series related to the Fibonacci numbers is a fascinating and fertile topic that has many investigations during the last half-century. For example, Good \cite{good} showed the following remarkable formula:
\begin{equation}\label{eq 1.11}
\sum_{n = 0}^{+ \infty} \frac{1}{F_{2^n}} = \frac{7 - \sqrt{5}}{2}
\end{equation}
solving a problem proposed by Millin \cite{mill}. Hoggatt and Bicknell \cite{hogg} obtained a closed formula for the more general infinite series $\sum_{n = 0}^{+ \infty} \frac{1}{F_{k 2^n}}$ ($k \in \N^*$), which can be simplified as:
\begin{equation}\label{eq 2.10}
\sum_{n = 0}^{+ \infty} \frac{1}{F_{k 2^n}} = \frac{1}{F_k} + \frac{1}{F_{2 k}} + \frac{1}{F_{2 k} \Phi^{2 k}}
\end{equation} 
On the other hand, we have the following well-known formulas:
\begin{multline}\label{eq 1.12}
\sum_{n = 1}^{+ \infty} \frac{1}{F_n F_{n + 2}} = 1 ~,~ \sum_{n = 1}^{+ \infty} \frac{1}{F_{2 n - 1} F_{2 n + 1}} = \frac{\sqrt{5} - 1}{2} ~,~ \sum_{n = 1}^{+ \infty} \frac{1}{F_{2 n} F_{2 n + 2}} = \frac{3 - \sqrt{5}}{2} ~,~ \\
\sum_{n = 1}^{+ \infty} \frac{(-1)^n}{F_n F_{n + 1}} = \frac{1 - \sqrt{5}}{2} ~,~ \sum_{n = 1}^{+ \infty} \frac{(-1)^n}{F_n F_{n + 2}} = 2 - \sqrt{5} ~~~~~~~~~~ 
\end{multline}
In this paper, we investigate two types of infinite Fibonacci related series. The first one consists of the series of the form:
$$
\sum_{n = 1}^{+ \infty} (-1)^{\varepsilon_n} \frac{F_{u_{n + k} - u_n}}{F_{u_n} F_{u_{n + k}}} ,
$$
where $k$ is a positive integer, ${(u_n)}_n$ is a sequence of positive integers tending to infinity with $n$ and ${(\varepsilon_n)}_n$ is a sequence of integers depending on ${(u_n)}_n$. The obtained results concerning this type of series generalize, in particular, Formulas \eqref{eq 2.10} and \eqref{eq 1.12}. The second type of series which we consider consists of the series of the form:
$$
\sum_{n = 1}^{+ \infty} \frac{(-1)^{\varepsilon_n}}{F_{u_n} \Phi^{u_n}} ,
$$
where ${(u_n)}_n$ is an increasing sequence of positive integers and ${(\varepsilon_n)}_n$ is a sequence of integers depending on ${(u_n)}_n$. We show in particular that some of such series can be transformed on series with rational terms. 

\section{The first type of series}\label{sec 2}
In this section, we deal with series of the form $\sum_{n \geq 1} (-1)^{\varepsilon_n} \frac{F_{u_{n + k} - u_n}}{F_{u_n} F_{u_{n + k}}}$, where $k$ is a positive integer, ${(u_n)}_n$ is a sequence of positive integers tending to infinity with $n$ and ${(\varepsilon_n)}_n$ is a sequence of integers depending on ${(u_n)}_n$. We begin with the following general result.
\begin{thm}\label{thm 2.1}
Let ${(u_n)}_{n \geq 1}$ be a sequence of positive integers, tending to infinity with $n$, and let $k$ be a positive integer. Then, for any $t \in \Z^*$, we have:
\begin{equation}\label{eq 2.1}
\sum_{n = 1}^{+ \infty} (-1)^{u_n} \frac{F_{u_{n + k} - u_n}}{F_{u_n} F_{u_{n + k}}} = \frac{1}{F_t} \left(\sum_{n = 1}^{k} \frac{F_{u_n + t}}{F_{u_n}} - k \Phi^t\right) 
\end{equation}
\end{thm} 
To prove this theorem, we need the following lemma.
\begin{lemma}\label{lemma 2.2}
Let ${(x_n)}_{n \geq 1}$ be a convergent real sequence and let $x \in \R$ be its limit. Then, for all $k \in \N$, we have:
$$
\sum_{n = 1}^{+ \infty} \left(x_{n + k} - x_n\right) = k x - \sum_{n = 1}^{k} x_n .
$$
\end{lemma} 
\begin{proof}
Let $k \in \N$ be fixed. For any positive integer $N$, we have:
\begin{eqnarray*}
\sum_{n = 1}^{N} \left(x_{n + k} - x_n\right) & = & \sum_{n = 1}^{N} \sum_{i = 1}^{k} \left(x_{n + i} - x_{n + i - 1}\right) \\
& = & \sum_{i = 1}^{k} \sum_{n = 1}^{N} \left(x_{n + i} - x_{n + i - 1}\right) \\
& = & \sum_{i = 1}^{k} \left(x_{N + i} - x_i\right) \\
& = & \sum_{i = 1}^{k} x_{N + i} - \sum_{i = 1}^{k} x_i .
\end{eqnarray*}
The formula of the lemma immediately follows by tending $N$ to infinity.
\end{proof}
\begin{proof}[Proof of Theorem \ref{thm 2.1}]
Let $t \in \Z^*$ be fixed. By applying the generalized subtraction formula \eqref{eq 1.8} for the triplet $(u_{n + k} + t , u_n + t , - t)$, we obtain:
$$
F_{- t} F_{u_{n + k} - u_n} = (-1)^{u_n + t} \left(F_{u_{n + k} + t} F_{u_n} - F_{u_n + t} F_{u_{n + k}}\right) .
$$
Then, by dividing the two sides of this last equality by $(-1)^{u_n} F_{- t} F_{u_n} F_{u_{n + k}}$, we obtain (since $F_{- t} = (-1)^{t + 1} F_t$ according to \eqref{eq 1.2}):
$$
(-1)^{u_n} \frac{F_{u_{n + k} - u_n}}{F_{u_n} F_{u_{n + k}}} = - \frac{1}{F_t} \left(\frac{F_{u_{n + k} + t}}{F_{u_{n + k}}} - \frac{F_{u_n + t}}{F_{u_n}}\right) .
$$
By taking the sum from $n = 1$ to infinity on the two sides of this last equality
and then by applying Lemma \ref{lemma 2.2} for $x_n = \frac{F_{u_n + t}}{F_{u_n}}$, which tends to $\Phi^t$ as $n$ tends to infinity (according to the Binet Formula \eqref{eq 1.3}), we conclude that:
$$
\sum_{n = 1}^{+ \infty} (-1)^{u_n} \frac{F_{u_{n + k} - u_n}}{F_{u_n} F_{u_{n + k}}} = \frac{1}{F_t} \left(\sum_{n = 1}^{k} \frac{F_{u_n + t}}{F_{u_n}} - k \Phi^t\right) ,
$$
as required. Theorem \ref{thm 2.1} is proved.
\end{proof}

\noindent\textbf{Remark.} As we will see below, the (non-vanishing) integer parameter $t$ appearing in the right-hand side of Formula \eqref{eq 2.1} plays an important role in the applications. Indeed, in each situation, we must choose it properly in order to obtain the simplest possible formula. 

From Theorem \ref{thm 2.1}, we deduce the following result which is already pointed out by Hu et al. \cite{hu} in the more general case of the generalized Fibonacci sequences.
\begin{coll}\label{coll 2.3}
Let ${(u_n)}_{n \geq 1}$ be a sequence of positive integers, tending to infinity with $n$. Then, we have:
\begin{equation}\label{eq 2.2}
\sum_{n = 1}^{+ \infty} (-1)^{u_n} \frac{F_{u_{n + 1} - u_n}}{F_{u_n} F_{u_{n + 1}}} = \frac{(-1)^{u_1}}{F_{u_1} \Phi^{u_1}}
\end{equation}
\end{coll}
\begin{proof}
It suffices to take in Formula \eqref{eq 2.1} of Theorem \ref{thm 2.1}: $k = 1$ and $t = - u_1$ (recall also that $F_{- u_1} = (-1)^{u_1 + 1} F_{u_1}$, according to Formula \eqref{eq 1.2}).  
\end{proof}
If ${(u_n)}_{n \geq 1}$ is an arithmetic sequence of positive integers, Theorem \ref{thm 2.1} gives the following result constituting a generalization of the both results by Brousseau \cite{brou} (corresponding to $u_n = n$) and by Popov \cite{popo} (corresponding to $k = 1$ and $r$ even).
\begin{coll}\label{coll 2.4}
Let ${(u_n)}_{n \geq 1}$ be an increasing arithmetic sequence of positive integers and let $r$ be its common difference. Then, for any positive integer $k$, we have:
\begin{equation}\label{eq 2.3}
\sum_{n = 1}^{+ \infty} \frac{(-1)^{r (n - 1)}}{F_{u_n} F_{u_{n + k}}} = \frac{(-1)^{u_1}}{F_{k r}} \left(\sum_{n = 1}^{k} \frac{F_{u_n + 1}}{F_{u_n}} - k \Phi\right)
\end{equation}
In particular, we have:
\begin{equation}\label{eq 2.4}
\sum_{n = 1}^{+ \infty} \frac{(-1)^{r (n - 1)}}{F_{u_n} F_{u_{n + 1}}} = \frac{1}{F_r F_{u_1} \Phi^{u_1}}
\end{equation}
\end{coll} 
\begin{proof}
To obtain Formula \eqref{eq 2.3}, it suffices to take $t = 1$ in Formula \eqref{eq 2.1} of Theorem \ref{thm 2.1} and use that $u_n = r (n - 1) + u_1$ ($\forall n \geq 1$). To obtain Formula \eqref{eq 2.4}, we simply set $k = 1$ in formula \eqref{eq 2.3} and use that $F_{u_1 + 1} - \Phi F_{u_1} = \Phibar^{u_1} = \left(\frac{-1}{\Phi}\right)^{u_1}$ (according to the Binet Formula \eqref{eq 1.3}). This completes the proof. 
\end{proof}

\subsection*{Some numerical examples}
\begin{itemize}
\item All the formulas of \eqref{eq 1.12}, apart from the first one, constitute special cases of Formula \eqref{eq 2.3} of Corollary \ref{coll 2.4}.   
\item Let $k$ be a positive integer and $a \geq 2$ be an integer. By taking in Formula \eqref{eq 2.2} of Corollary \ref{coll 2.3}: $u_n = k a^n$ ($\forall n \geq 1$), we obtain the following formula:
\begin{equation}\label{eq 2.9}
\sum_{n = 1}^{+ \infty} \frac{F_{(a - 1) k a^n}}{F_{k a^n} F_{k a^{n + 1}}} = \frac{1}{F_{k a} \Phi^{k a}}
\end{equation}
which can be considered as a generalization of Formula \eqref{eq 2.10} (we obtain \eqref{eq 2.10} by taking $a = 2$ in \eqref{eq 2.9}).

From Formula \eqref{eq 2.10}, we deduce that $\sum_{n = 0}^{+ \infty} \frac{1}{F_{k 2^n}} \in \Q(\sqrt{5})$ (for any positive integer $k$). But except the geometric sequences with common ration $2$, we don't know any other ``regular'' sequence ${(u_n)}_{n \in \N}$ of positive integers, satisfying the property that $\sum_{n = 0}^{+ \infty} \frac{1}{F_{u_n}} \in \Q(\sqrt{5})$. More precisely, we propose the following open question:
\begin{quote}
\noindent \textbf{Open question.} \emph{Is there any linear recurrence sequence ${(u_n)}_{n \in \N}$ of positive integers, which is not a geometric sequence with common ratio $2$ and which satisfies the property that
$$
\sum_{n = 0}^{+ \infty} \frac{1}{F_{u_n}} \in \Q(\sqrt{5}) ?
$$
}
\end{quote}
Next, by taking $a = 3$ in Formula \eqref{eq 2.9}, we deduce (according to Formula \eqref{eq 1.7}) the following:
\begin{equation}\label{eq c}
\sum_{n = 1}^{+ \infty} \frac{L_{k 3^n}}{F_{k 3^{n + 1}}} = \frac{1}{F_{3 k} \Phi^{3 k}}
\end{equation}
By taking $k = 1$ in Formula \eqref{eq c}, we deduce (after some calculations) the formula:
\begin{equation}\label{eq 2.11}
\sum_{n = 0}^{+ \infty} \frac{L_{3^n}}{F_{3^{n + 1}}} = \frac{\sqrt{5} - 1}{2}
\end{equation}
\item By taking in Formula \eqref{eq 2.1} of Theorem \ref{thm 2.1}: $u_n = F_n$, $k = 2$ and $t = 1$, we obtain the following:
\begin{equation}\label{eq 2.12}
\sum_{n = 1}^{+ \infty} (-1)^{F_n} \frac{F_{F_{n + 1}}}{F_{F_n} F_{F_{n + 2}}} = 1 - \sqrt{5}
\end{equation}
Next, by taking in Formula \eqref{eq 2.1} of Theorem \ref{thm 2.1}: $u_n = F_n$, $k = 1$ and $t = -1$, we obtain the following:
\begin{equation}\label{eq 2.13}
\sum_{n = 1}^{+ \infty} (-1)^{F_n} \frac{F_{F_{n - 1}}}{F_{F_n} F_{F_{n + 1}}} = \frac{1 - \sqrt{5}}{2}
\end{equation}
\end{itemize}
In what follows, we will discover that for the sequences ${(u_n)}_{n \geq 1}$ of positive integers for which any term $u_n$ have the same parity with $n$, it is even possible to determine the sum of the series $\sum_{n \geq 1} \frac{F_{u_{n + k} - u_n}}{F_{u_n} F_{u_{n + k}}}$ (where $k$ is a fixed even positive integer). We have the following:
\begin{thm}\label{thm 2.5}
Let ${(u_n)}_{n \geq 1}$ be a sequence of positive integers, tending to infinity with $n$, and let $k$ be a positive integer. Then, we have:
\begin{equation}\label{eq 2.14}
\sum_{n = 1}^{+ \infty} (-1)^{u_n - n} \frac{F_{u_{n + 2 k} - u_n}}{F_{u_n} F_{u_{n + 2 k}}} = \sum_{n = 1}^{k} (-1)^{u_{2 n - 1} + 1} \frac{F_{u_{2 n} - u_{2 n - 1}}}{F_{u_{2 n}} F_{u_{2 n - 1}}} 
\end{equation}
In particular, if $u_n$ have the same parity with $n$ (for any positive integer $n$), then we have:
\begin{equation}\label{eq 2.15}
\sum_{n = 1}^{+ \infty} \frac{F_{u_{n + 2 k} - u_n}}{F_{u_n} F_{u_{n + 2 k}}} = \sum_{n = 1}^{k} \frac{F_{u_{2 n} - u_{2 n - 1}}}{F_{u_{2 n}} F_{u_{2 n - 1}}} 
\end{equation}
\end{thm}
\begin{proof}
By applying Formula \eqref{eq 2.1} of Theorem \ref{thm 2.1} for the sequence ${(u_{2 n})}_{n \geq 1}$ and for $t = 1$, we obtain:
\begin{equation}\label{eq 2.16}
\sum_{n = 1}^{+ \infty} (-1)^{u_{2 n}} \frac{F_{u_{2 n + 2 k} - u_{2 n}}}{F_{u_{2 n}} F_{u_{2 n + 2 k}}} = \sum_{n = 1}^{k} \frac{F_{u_{2 n} + 1}}{F_{u_{2 n}}} - k \Phi
\end{equation}
Next, by applying Formula \eqref{eq 2.1} of Theorem \ref{thm 2.1} for the sequence ${(u_{2 n - 1})}_{n \geq 1}$ and for $t = 1$, we obtain:
\begin{equation}\label{eq 2.17}
\sum_{n = 1}^{+ \infty} (-1)^{u_{2 n - 1}} \frac{F_{u_{2 n + 2 k - 1} - u_{2 n - 1}}}{F_{u_{2 n - 1}} F_{u_{2 n + 2 k - 1}}} = \sum_{n = 1}^{k} \frac{F_{u_{2 n - 1} + 1}}{F_{u_{2 n - 1}}} - k \Phi
\end{equation}
Now, by subtracting \eqref{eq 2.17} from \eqref{eq 2.16}, we get:
\begin{multline*}
\sum_{n = 1}^{+ \infty} \left((-1)^{u_{2 n}} \frac{F_{u_{2 n + 2 k} - u_{2 n}}}{F_{u_{2 n}} F_{u_{2 n + 2 k}}} - (-1)^{u_{2 n - 1}} \frac{F_{u_{2 n + 2 k - 1} - u_{2 n - 1}}}{F_{u_{2 n - 1}} F_{u_{2 n + 2 k - 1}}} \right) \\
= \sum_{n = 1}^{k} \left(\frac{F_{u_{2 n} + 1}}{F_{u_{2 n}}} - \frac{F_{u_{2 n - 1} + 1}}{F_{u_{2 n - 1}}} \right) ,  
\end{multline*}
which we can write as:
$$
\sum_{n = 1}^{+ \infty} (-1)^{u_n - n} \frac{F_{u_{n + 2 k} - u_n}}{F_{u_n} F_{u_{n + 2 k}}} = \sum_{n = 1}^{k} \frac{F_{u_{2 n} + 1} F_{u_{2 n - 1}} - F_{u_{2 n}} F_{u_{2 n - 1} + 1}}{F_{u_{2 n}} F_{u_{2 n - 1}}} .
$$
Finally, since $F_{u_{2 n} + 1} F_{u_{2 n - 1}} - F_{u_{2 n}} F_{u_{2 n - 1} + 1} = (-1)^{u_{2 n - 1} + 1} F_{u_{2 n} - u_{2 n - 1}}$ (according to Formula \eqref{eq 1.8} applied for the triplet $(u_{2 n} + 1 , u_{2 n - 1} + 1 , -1)$), we conclude that:
$$
\sum_{n = 1}^{+ \infty} (-1)^{u_n - n} \frac{F_{u_{n + 2 k} - u_n}}{F_{u_n} F_{u_{n + 2 k}}} = \sum_{n = 1}^{k} (-1)^{u_{2 n - 1} + 1} \frac{F_{u_{2 n} - u_{2 n - 1}}}{F_{u_{2 n}} F_{u_{2 n - 1}}} ,
$$
which is the first formula of the theorem. 

The second formula of the theorem is an immediate consequence of the previous one. The proof is achieved.  
\end{proof}
If ${(u_n)}_{n \geq 1}$ is an arithmetic sequence of positive integers then Theorem \ref{thm 2.5} reduces to the following corollary constituting a generalization of a result by Brousseau \cite{brou} (which actually corresponds to the case $u_n  = n$). 
\begin{coll}\label{coll 2.6}
Let ${(u_n)}_{n \geq 1}$ be an increasing arithmetic sequence of positive integers and let $r$ be its common difference. Then, for any positive integer $k$, we have:
\begin{equation}\label{eq 2.18}
\sum_{n = 1}^{+ \infty} \frac{(-1)^{(r - 1) (n - 1)}}{F_{u_n} F_{u_{n + 2 k}}} = \frac{F_r}{F_{2 k r}} \sum_{n = 1}^{k} \frac{1}{F_{u_{2 n}} F_{u_{2 n - 1}}}
\end{equation}
In particular, we have:
\begin{equation}\label{eq e}
\sum_{n = 1}^{+ \infty} \frac{(-1)^{(r - 1) (n - 1)}}{F_{u_n} F_{u_{n + 2}}} = \frac{1}{F_{u_1} F_{u_2} L_r}
\end{equation}
\end{coll}
\begin{proof}
To establish Formula \eqref{eq 2.18}, it suffices to apply Theorem \ref{thm 2.5} together with the formula $u_n = r (n - 1) + u_1$ ($\forall n \geq 1$). To establish Formula \eqref{eq e}, we take $k = 1$ in \eqref{eq 2.18} and we use in addition Formula \eqref{eq 1.7}. 
\end{proof}

\noindent\textbf{Remark.} Let ${(u_n)}_{n \in \N}$ be an increasing arithmetic sequence of natural numbers and let $r$ be its common difference. By Corollary \ref{coll 2.4}, we know a closed form of the sum $\sum_{n = 1}^{+ \infty} \frac{(-1)^{r (n - 1)}}{F_{u_n} F_{u_{n + k}}}$ ($k \in \N^*$) and by Corollary \ref{coll 2.6}, we know a closed form of the sum $\sum_{n = 1}^{+ \infty} \frac{(-1)^{(r - 1) (n - 1)}}{F_{u_n} F_{u_{n + k}}}$ ($k$ an even positive integer). But if $k$ is an odd positive integer, the closed form of the sum $\sum_{n = 1}^{+ \infty} \frac{(-1)^{(r - 1)(n - 1)}}{F_{u_n} F_{u_{n + k}}}$ is still unknown even in the particular case ``$u_n = n$''. However, we shall prove in what follows that if $u_0 = 0$, there is a relationship between the sums $\sum_{n = 1}^{+ \infty} \frac{(-1)^{(r - 1)(n - 1)}}{F_{u_n} F_{u_{n + k}}}$, where $k$ lies in the set of the odd positive integers. We have the following:   
\begin{thm}\label{thm a}
Let
$$
S_{r , k} := \sum_{n = 1}^{+ \infty} \frac{(-1)^{(r - 1)(n - 1)}}{F_{r n} F_{r (n + k)}} ~~~~~~~~~~ (\forall r , k \in \N^*) .
$$
Then, for any positive integer $r$ and any odd positive integer $k$, we have:
\begin{equation}\label{eq d}
S_{r , k} = \frac{F_r}{F_{r k}} \left(S_{r , 1} + (-1)^r \sum_{n = 1}^{(k - 1)/2} \frac{1}{F_{2 n r} F_{(2 n + 1) r}}\right)
\end{equation}
\end{thm}
\begin{proof}
Let $r$ and $k$ be positive integers and suppose that $k$ is odd. Because the formula of the theorem is trivial for $k = 1$, we can assume that $k \geq 3$. We have:
\begin{eqnarray*}
S_{r , 1} - \frac{F_{r k}}{F_r} S_{r , k} & = & \sum_{n = 1}^{+ \infty} \frac{(-1)^{(r - 1)(n - 1)}}{F_{r n} F_{r (n + 1)}} - \frac{F_{r k}}{F_r} \sum_{n = 1}^{+ \infty} \frac{(-1)^{(r - 1)(n - 1)}}{F_{r n} F_{r (n + k)}} \\
& = & \sum_{n = 1}^{+ \infty} (-1)^{(r - 1)(n - 1)} \frac{F_r F_{r (n + k)} - F_{r k} F_{r (n + 1)}}{F_r F_{r n} F_{r (n + 1)} F_{r (n + k)}} .
\end{eqnarray*}
But according to Formula \eqref{eq 1.8} (applied for the triplet $(r , r k , r n)$), we have:
\begin{eqnarray*}
F_r F_{r (n + k)} - F_{r k} F_{r (n + 1)} & = & (-1)^{r k} F_{r n} F_{r (1 - k)} \\
& = & (-1)^{r k} F_{r n} (-1)^{r (k - 1) + 1} F_{r (k - 1)} ~~~~~~~~ \text{(by \eqref{eq 1.2})} \\
& = & (-1)^{r + 1} F_{r n} F_{r (k - 1)} . 
\end{eqnarray*}
Using this, it follows that:
\begin{eqnarray*}
S_{r , 1} - \frac{F_{r k}}{F_r} S_{r , k} & = & (-1)^{r + 1} \frac{F_{r (k - 1)}}{F_r} \sum_{n = 1}^{+ \infty} \frac{(-1)^{(r - 1)(n - 1)}}{F_{r (n + 1)} F_{r (n + k)}} \\
& = & (-1)^{r + 1} \frac{F_{r (k - 1)}}{F_r} \sum_{n = 1}^{+ \infty} \frac{(-1)^{(r - 1)(n - 1)}}{F_{u_n} F_{u_{n + k - 1}}} , 
\end{eqnarray*}
where we have put $u_n := r (n + 1)$ ($\forall n \geq 1$). \\
Next, because $(k - 1)$ is even (since $k$ is supposed odd), it follows by Corollary \ref{coll 2.6} that:
\begin{eqnarray*}
S_{r , 1} - \frac{F_{r k}}{F_r} S_{r , k} & = & (-1)^{r + 1} \frac{F_{r (k - 1)}}{F_r} \times \frac{F_r}{F_{(k - 1) r}} \sum_{n = 1}^{(k - 1)/2} \frac{1}{F_{u_{2 n}} F_{u_{2 n - 1}}} \\
& = & (-1)^{r + 1} \sum_{n = 1}^{(k - 1)/2} \frac{1}{F_{2 n r} F_{(2 n + 1) r}} .
\end{eqnarray*}
The formula of the theorem follows.
\end{proof}

\pagebreak

\noindent\textbf{Remarks.}
\begin{enumerate}
\item The particular case of Formula \eqref{eq d} corresponding to $r = 1$ is already established by Rabinowitz \cite{rabi}.
\item If $r$ is an odd positive integer, Jeannin \cite{jean2} established an expression of $S_{r , 1}$ in terms of the values of the Lambert series. 
\end{enumerate}
\section{The second type of series}
In this section, we deal with series of the form $\sum_{n \geq 1} \frac{(-1)^{\varepsilon_n}}{F_{u_n} \Phi^{u_n}}$, where ${(u_n)}_n$ is an increasing sequence of positive integers and ${(\varepsilon_n)}_n$ is a sequence of integers depending on ${(u_n)}_n$. By grouping terms, we show that it is possible to transform some such series to series with rational terms. As we will precise later, some results of this section can be deduced from the results of the previous one by tending the parameter $k$ to infinity. We begin with the following:
\begin{thm}\label{thm 3.1}
Let ${(u_n)}_{n \geq 1}$ be an increasing sequence of positive integers. Then, we have:
\begin{equation}\label{eq 3.1}
\sum_{n = 1}^{+ \infty} \frac{(-1)^{u_n - n}}{F_{u_n} \Phi^{u_n}} = \sum_{n = 1}^{+ \infty} (-1)^{u_{2 n - 1} + 1} \frac{F_{u_{2 n} - u_{2 n - 1}}}{F_{u_{2 n}} F_{u_{2 n - 1}}}
\end{equation}
In particular, if $u_n$ have the same parity with $n$ (for any positive integer $n$), then we have:
\begin{equation}\label{eq 3.2}
\sum_{n = 1}^{+ \infty} \frac{1}{F_{u_n} \Phi^{u_n}} = \sum_{n = 1}^{+ \infty} \frac{F_{u_{2 n} - u_{2 n - 1}}}{F_{u_{2 n}} F_{u_{2 n - 1}}}
\end{equation}
\end{thm}
\begin{proof}
The increase of ${(u_n)}_n$ ensures the convergence of the two series in \eqref{eq 3.1}. By grouping terms, we have:
\begin{eqnarray*}
\sum_{n = 1}^{+ \infty} \frac{(-1)^{u_n - n}}{F_{u_n} \Phi^{u_n}} & = & \sum_{n = 1}^{+ \infty} \left(\frac{(-1)^{u_{2 n} - 2 n}}{F_{u_{2 n}} \Phi^{u_{2 n}}} + \frac{(-1)^{u_{2 n - 1} - (2 n - 1)}}{F_{u_{2 n - 1}} \Phi^{u_{2 n - 1}}}\right) \\
& = & \sum_{n = 1}^{+ \infty} (-1)^{u_{2 n - 1} + 1} \frac{\Phi^{u_{2 n} - u_{2 n - 1}} F_{u_{2 n}} + (-1)^{u_{2 n} - u_{2 n - 1} + 1} F_{u_{2 n - 1}}}{F_{u_{2 n}} F_{u_{2 n - 1}} \Phi^{u_{2 n}}} \\
& \!\!\!\!\!\!\!\!\!\!\!\!\!\!\!\!\!\!\!\!\!\!\!\!\!\!\!\!= & \!\!\!\!\!\!\!\!\!\!\!\!\!\! \sum_{n = 1}^{+ \infty} (-1)^{u_{2 n - 1} + 1} \frac{\Phi^{u_{2 n}} F_{u_{2 n} - u_{2 n - 1}}}{F_{u_{2 n}} F_{u_{2 n - 1}} \Phi^{u_{2 n}}} ~~~~\text{(according to Formula \eqref{eq 1.9})} \\
\\
& \!\!\!\!\!\!\!\!\!\!\!\!\!\!\!\!\!\!\!\!\!\!\!\!\!\!\!\!= & \!\!\!\!\!\!\!\!\!\!\!\!\!\! \sum_{n = 1}^{+ \infty} (-1)^{u_{2 n - 1} + 1} \frac{F_{u_{2 n} - u_{2 n - 1}}}{F_{u_{2 n}} F_{u_{2 n - 1}}} ,
\end{eqnarray*}
which confirms the first formula of the theorem.

The second formula of the theorem is an immediate consequence of the first one. The proof is complete.
\end{proof}

\noindent\textbf{Remark.} We can also prove Theorem \ref{thm 3.1} by tending $k$ to infinity in Formulas \eqref{eq 2.14} and \eqref{eq 2.15} of Theorem \ref{thm 2.5}. To do so, we must previously remark that for any positive integer $n$, we have:
$$
\lim_{k \rightarrow + \infty} \frac{F_{u_{n + 2 k} - u_n}}{F_{u_{n + 2 k}}} = \frac{1}{\Phi^{u_n}}
$$
(according to the Binet Formula \eqref{eq 1.3}) and that the series of functions $\sum_{n = 1}^{+ \infty} (-1)^{u_n - n} \frac{F_{u_{n + 2 k} - u_n}}{F_{u_n} F_{u_{n + 2 k}}}$ (from $\N^*$ to $\R$) converges uniformly on $\N^*$.

According to the closed formulas of $F_n$ and $L_n$ (see Formulas \eqref{eq 1.3} and \eqref{eq 1.6}), it is immediate that $\lim_{n \rightarrow + \infty} \frac{L_n}{F_n} = \sqrt{5}$. So the approximations $\sqrt{5} \simeq \frac{L_n}{F_n}$ $(n \geq 1)$ are increasingly better when $n$ increases. From Theorem \ref{thm 3.1}, we derive a curious formula in which the sum of the errors of the all approximations $\sqrt{5} \simeq \frac{L_n}{F_n}$ ($n \geq 1$) is transformed to a series with rational terms. We have the following:
\begin{coll}\label{coll 3.2}
We have:
\begin{equation}\label{eq 3.3}
\sum_{n = 1}^{+ \infty} \left|\sqrt{5} - \frac{L_n}{F_n}\right| = 2 \sum_{n = 1}^{+ \infty} \frac{1}{F_n \Phi^n} = 2 \sum_{n = 1}^{+ \infty} \frac{1}{F_{2 n} F_{2 n - 1}}
\end{equation}
\end{coll}  
\begin{proof}
According to Formulas \eqref{eq 1.3} and \eqref{eq 1.6}, we have for any positive integer $n$:
\begin{multline*}
\left\vert\sqrt{5} - \frac{L_n}{F_n}\right\vert = \frac{\left\vert\sqrt{5} F_n - L_n\right\vert}{F_n} = \frac{\left\vert\left(\Phi^n - \Phibar^n\right) - \left(\Phi^n + \Phibar^n\right)\right\vert}{F_n} = \frac{\left\vert - 2 \Phibar^n\right\vert}{F_n} \\ = \frac{2}{F_n \Phi^n} ~~~~~~~~\text{(since $\Phibar = - \frac{1}{\Phi}$).}
\end{multline*}
So, it follows that:
$$
\sum_{n = 1}^{+ \infty} \left\vert \sqrt{5} - \frac{L_n}{F_n}\right\vert = 2 \sum_{n = 1}^{+ \infty} \frac{1}{F_n \Phi^n} ,
$$
which is the first equality of \eqref{eq 3.3}.

The second equality of \eqref{eq 3.3} follows from Theorem \ref{thm 3.1} by taking $u_n = n$ ($\forall n \geq 1$). The proof is complete.
\end{proof}
Further, it is known that for any positive integer $n$, the $n$\textsuperscript{th} convergent of the regular continued fraction expansion of the number $\sqrt{5}$ is equal to the irreducible rational fraction $r_n := \frac{L_{3 n}/2}{F_{3 n}/2}$. So, by repeating the proof of Corollary \ref{coll 3.2}, replacing the index $n$ by $3 n$, we immediately obtain the following:
\begin{coll}\label{coll a}
We have:
\begin{equation}\label{eq b}
\sum_{r \in \Lambda} \left|\sqrt{5} - r\right| = 2 \sum_{n = 1}^{+ \infty} \frac{1}{F_{3 n} \Phi^{3 n}} = 4 \sum_{n = 1}^{+ \infty} \frac{1}{F_{6 n} F_{6 n - 3}}
\end{equation}
where $\Lambda$ denotes the set of the regular continued fraction convergents of the number $\sqrt{5}$. \hfill $\square$
\end{coll} 

\subsection*{Some other numerical applications}
By taking successively in Theorem \ref{thm 3.1}: $u_n = 2 n$, $u_n = 2 n - 1$ and then $u_n = 2 n + 1$, we respectively obtain the three following formulas:
\begin{equation}\label{eq 3.5}
\begin{split}
\sum_{n = 1}^{+ \infty} \frac{(-1)^{n + 1}}{F_{2 n} \Phi^{2 n}} = \sum_{n = 1}^{+ \infty} \frac{1}{F_{4 n} F_{4 n - 2}} ~,~ 
\sum_{n = 1}^{+ \infty} \frac{(-1)^{n + 1}}{F_{2 n - 1} \Phi^{2 n - 1}} = \sum_{n = 1}^{+ \infty} \frac{1}{F_{4 n - 1} F_{4 n - 3}} ~, \\
\sum_{n = 1}^{+ \infty} \frac{(-1)^{n + 1}}{F_{2 n + 1} \Phi^{2 n + 1}} = \sum_{n = 1}^{+ \infty} \frac{1}{F_{4 n + 1} F_{4 n - 1}}
\end{split}
\end{equation}
Remark that the addition (side to side) of the two last formulas of \eqref{eq 3.5} gives the second formula of \eqref{eq 1.12}.

Now, by applying another technique of grouping terms, we are going to show that some other series of the form $\sum_{n \geq 1} \frac{1}{F_{u_n} \Phi^{u_n}}$ (in which $u_n$ have not always the same parity with $n$) can be also transformed to series with rational terms; precisely to series of reciprocals of Fibonacci numbers. We have the following:
\begin{thm}\label{thm 3.3}
For any integer $r \geq 2$, we have:
\begin{equation}\label{eq a}
\sum_{n = 0}^{+ \infty} \frac{1}{F_{2^r n + 2^{r - 1}} \Phi^{2^r n + 2^{r - 1}}} = \sum_{n = 1}^{+ \infty} \frac{1}{F_{2^r n}}
\end{equation}
\end{thm} 
\begin{proof}
Let $r \geq 2$ be an integer. From the trivial equality of sets:
$$
\left\{2^r n + 2^{r - 1} ~,~ n \in \N\right\} = \left\{2^{r - 1} n ~,~ n \in \N^*\right\} \setminus \left\{2^r n ~,~ n \in \N^*\right\} ,
$$
we have:
\begin{eqnarray*}
\sum_{n = 0}^{+ \infty} \frac{1}{F_{2^r n + 2^{r - 1}} \Phi^{2^r n + 2^{r - 1}}} & = & \sum_{n = 1}^{+ \infty} \frac{1}{F_{2^{r - 1} n} \Phi^{2^{r - 1} n}} - \sum_{n = 1}^{+ \infty} \frac{1}{F_{2^r n} \Phi^{2^r n}} \\
& = & \sum_{n = 1}^{+ \infty} \left(\frac{1}{F_{2^{r - 1} n} \Phi^{2^{r - 1} n}} - \frac{1}{F_{2^r n} \Phi^{2^r n}}\right) \\
& = & \sum_{n = 1}^{+ \infty} \frac{L_{2^{r - 1} n} \Phi^{2^{r - 1} n} - 1}{F_{2^r n} \Phi^{2^r n}}
\end{eqnarray*}
(since $F_{2^r n} = F_{2^{r - 1} n} L_{2^{r - 1} n}$, according to \eqref{eq 1.7}). But since $L_{2^{r - 1} n} \Phi^{2^{r - 1} n} - 1 = (\Phi^{2^{r - 1} n} + \Phibar^{2^{r - 1} n}) \Phi^{2^{r - 1} n} - 1 = \Phi^{2^r n} + (\Phi \Phibar)^{2^{r - 1} n} - 1 = \Phi^{2^r n}$ (according to \eqref{eq 1.6} and to that $\Phibar = - \frac{1}{\Phi}$), we conclude that:
$$
\sum_{n = 0}^{+ \infty} \frac{1}{F_{2^r n + 2^{r - 1}} \Phi^{2^r n + 2^{r - 1}}} = \sum_{n = 1}^{+ \infty} \frac{1}{F_{2^r n}} ,
$$
as required. The theorem is proved.
\end{proof}

\noindent\textbf{Remark.} Taking $r = 2$ in Formula \eqref{eq a} of Theorem \ref{thm 3.3} gives the following:
\begin{equation}\label{eq f}
\sum_{n = 0}^{+ \infty} \frac{1}{F_{4 n + 2} \Phi^{4 n + 2}} = \sum_{n = 1}^{+ \infty} \frac{1}{F_{4 n}} 
\end{equation}
Note that this last formula was already pointed out by Melham and Shannon \cite{melh} who proved it by summing both sides of Formula \eqref{eq 2.10} over $k$, lying in the set of the odd positive integers.

\end{document}